\documentclass[a4paper,12pt]{article}
\usepackage{amsmath,amssymb}
\usepackage{parskip,url}
\usepackage{cancel }
\usepackage{comment}
\usepackage[pdfview=Fit]{hyperref}
\usepackage{mathtools}
\hypersetup{bookmarkstype=none}
\newtheorem{theorem}{Theorem}[section]
\newtheorem{lemma}[theorem]{Lemma}
\newtheorem{proposition}[theorem]{Proposition}
\newtheorem{corollary}[theorem]{Corollary}

\newcommand{\id}{\mbox{id}}

\newcommand{\Sym}{\operatorname{Sym}}

\newcommand{\Q}{\operatorname{Q}}

\newcommand{\V}{\tilde{v}}
\newcommand{\W}{\tilde{w}}

\newcommand{\FM}{\operatorname{FM}}

\newenvironment{proof}{\par\noindent{\bf Proof.}}{$\qed$\par\bigskip}
\newcommand{\qed}{\enspace\vrule  height6pt  width4pt  depth2pt}
\begin{document}

\title{Construction of a two unique product semigroup defined by permutation relations of quaternion type}

\author{Ferran Ced\'o\footnote{Research partially supported by a grant of MICIIN (Spain)
MTM2011-28992-C02-01.} \and Eric Jespers\footnote{Research supported
in part by Onderzoeksraad of Vrije Universiteit Brussel and Fonds
voor Wetenschappelijk Onderzoek (Belgium).} \and Georg Klein}
\date{}
\maketitle

\begin{abstract}
 For a regular representation $H\subseteq \Sym_n$ of the
generalized quaternion group of order $n=4k$, with $k\geq 2$, the
monoid $S_n(H)$ presented with generators $a_1,a_2,\dots ,a_n$ and
with relations $a_1a_2\cdots a_n=a_{\sigma(1)}a_{\sigma(2)}\cdots
a_{\sigma(n)}$, for all $\sigma\in H$, is investigated.
%
It is shown that $S_n(H)$  has the two unique product
property. As a consequence, for any field $K$,  the monoid
algebra $K[S_n(H)]$ is a domain with trivial units which is
semiprimitive.
\end{abstract}


\section{Introduction}\label{intro}

The well-known unit conjecture for group algebras $KG$ of a group
$G$ over a field states that if $G$ is torsion-free then $KG$ only
has trivial units. That is, the only invertible elements of $KG$ are
$kg$, with $0\neq k\in K$ and $g\in G$. Very little progress has
been made on this conjecture. One of the few known results says that
if $G$ is a unique product group then the conjecture has a positive
answer. Recall that  a semigroup $S$ is said to be unique product
(u.p.) if  for any non-empty finite subsets  $C$ and $D$ of $S$
there is an element of $ C   D = \{c d \mid c \in C, d \in  d \}$
that is expressible uniquely as $c d$ with $c\in C$ and $d\in D$.
If, for any non empty finite subsets $C,
D$ with $|C| + |D|
> 2$, there exist at least two elements in $CD$ that are uniquely
expressible as $cd$, for $c\in C$ and $d\in D$ then $S$ is called a two unique product (t.u.p.) semigroup. Strojnowski
\cite{strojnowski} proved that u.p. groups are two unique product
groups. Using standard arguments, the unit conjecture
is proved easily for semigroup algebras of t.u.p. semigroups. Note
that such semigroups are necessarily cancellative. However,
semigroups which are u.p. are not necessarily t.u.p., a counter-example 
can be found in \cite[Example 10.13]{book-jan}. On the other
hand, in general, it is hard to verify whether  a (semi)group  is
t.u.p.. Rips and Segev, in \cite{ripssegev}, constructed examples of
torsion-free groups that are not t.u.p.,
and  Promislow, in
\cite{promislow}, constructed a simpler example.
So it turned out that not all torsion-free groups are t.u.p..

In the present article we construct a semigroup which is t.u.p.. The
idea comes from the finitely presented algebras defined by
homogeneous relations that are studied in
\cite{alt4algebra,alghomshort,altalgebra,symcyclic}. More precisely,
let $K$ be a field and let $H$ be a subgroup of the symmetric group
$\Sym_n$ of degree $n$, and consider the algebra $K\langle a_1 ,
\ldots , a_n \mid  a_1a_2\cdots a_n= a_{\sigma (1)}a_{\sigma
(2)}\cdots a_{\sigma (n)},\;  \sigma\in H\rangle$. Clearly, this is
the semigroup algebra $K[S_{n}(H)]$ over the monoid
$S_{n}(H)=\langle a_1 , \ldots , a_n \mid  a_1a_2\cdots a_n=
a_{\sigma (1)}a_{\sigma (2)}\cdots a_{\sigma (n)},\; \sigma\in
H\rangle$. In order to study these algebras it is important to study
the structure of the monoid $S_{n}(H)$. In the listed papers this
has been done in case, for example, $H$ is the symmetric or
alternating group, or an abelian group.  Furthermore in
\cite[Corollary~3.2]{symcyclic}, it is proved that $S_{n}(H)$ is
cancellative if and only if the stabilizers of $1$ and $n$ are
trivial, and in this case $S_n(H)$ is embedded in its universal
group $G_{n}(H)$, that is, the group defined by the ``same"
generators and relations. In particular, this happens if $H$ is a
regular subgroup of $\Sym_n$. n this paper we consider a
regular representation $H$ of the generalized quaternion group
$Q_{4k}$ of order $n=4k$ in $\Sym_n$, and we will show that
$S_{n}(H)$ is a t.u.p. semigroup. An immediate consequence is that
the algebra $K[S_n(H)]$ is a domain with trivial units, and thus the
algebra is semiprimititve. Note that $S_{n}(H)$ is embedded in
its universal group $G_{n}(H)$. We do not know whether $G_{n}(H)$
is a t.u.p. group.

\section{Preliminary results } \label{quaternionn}

In this section, we obtain properties of $S_{n}(H)$ in case $H$ is a
regular representation of the generalized quaternion group
$\Q_{4k} = \langle t, u \mid t^{2k}=u^4 = 1, t^k=u^2,  u^{-1}tu =
t^{-1} \rangle  $, with $n=4k \geq 8$. For simplicity, we  denote
$S_{n}(H)$ by $S_n$. In order to  represent $\Q_{4k}$ as a
regular subgroup $H$ of $\Sym_n$, we take its  Cayley representation
via left multiplication and we identify the  elements $t^{i}$ with $i+1$,
for $0\leq i\leq 2k-1$, and the elements $t^{i}u$ with $i+2k+1$, for
$0\leq i\leq 2k-1$. Thus $t$ corresponds to the product of $2$
disjoint cycles of length $2k$, and $u$ to the product of $k$
disjoint cycles of length $4$. Identifying $\Q_{4k}$ with this
representation, we get
$$t=(1,2, \ldots , 2k-1, 2k)(2k+1,2k+2, \dots
,4k-1, 4k)$$ and
\begin{eqnarray*}
u &=&(1,2k+1,1+k,2k+1+k)(2,4k,2+k,4k-k)(3,4k-1,3+k,4k-1-k)\cdots \\
& & \cdots(k,4k-(k-2),k+k,4k-(k-2)-k).
\end{eqnarray*}
The following two results are easy consequences of the definition of
$t$ and $u$ as permutations.
\begin{lemma}\label{other}
The generator $u$ must map an element from one of the cycles of $t$
to the other.
\end{lemma}
\begin{corollary}\label{disjoi}
For $\sigma \in H$, either
\begin{eqnarray*}
 \{ \sigma(1) , \sigma(2),  \ldots , \sigma(n/2) \} &=& \{ 1, 2, \ldots , n/2 \} \\
\text{ and }\ \ \ \ \{ \sigma(n/2+1) , \sigma(n/2+2),  \ldots ,
\sigma(n) \} &=& \{ n/2+1, n/2+2, \ldots , n \},
\end{eqnarray*}
or
\begin{eqnarray*}
 \{ \sigma(1) , \sigma(2),  \ldots , \sigma(n/2) \} &=& \{ n/2+1, n/2+2, \ldots , n \} \\
\text{ and }\ \ \ \ \{ \sigma(n/2+1) , \sigma(n/2+2),  \ldots ,
\sigma(n) \} &=&  \{ 1, 2, \ldots , n/2 \} .
\end{eqnarray*}
\end{corollary}
%

Since $H$ is a regular subgroup of $\Sym_n$ we have the
following result.

\begin{lemma}\label{stabilizerr}
For $\sigma \in H$ with $\sigma \neq \id$, and $1\leq i \leq n$, we
have $\sigma(i) \neq i$.
\end{lemma}

As a consequence of a result of Adjan\index{Adjan, S.} \cite{adyan1,
adyan2} (see also \cite[Theorem 3.1]{stallings}  and
\cite[Corollary~3.2]{symcyclic}), we obtain the following result.
\begin{proposition}\label{adjannn}
 $S_{n}(H)$ can be  embedded in its universal group $G_{n}(H)$. In particular, $S_n(H)$ is a cancellative monoid.
\end{proposition}

 Because of the choice made for the generators $t$ and $u$, the number of overlapping letters
of two words in the defining relations of $S_n(H)$ is at most $1$. For example, in case $k=2$, the following overlap is possible
for words in the defining relations of $S_n$.
$$
\overbracket{
 a_{1} a_{2} a_{3} a_{4} a_{5} a_{6} a_{7}  
\makebox[0pt][l]{$\displaystyle{\underbracket{\phantom{
 a_{8} a_{7}a_{6}a_{5} a_{2} a_{1} a_{4} a_{3} 
} }_{\text{   } } } $}   
a_{8}                                            
}^{\text{   }}    
a_{7}a_{6}a_{5} a_{2} a_{1} a_{4} a_{3} .$$   
This motivates the
following lemmas and their corollaries on permutations in $H$.

Throughout this section, we let $\FM_n$ denote the free monoid of
rank $n$ with basis $\{ x_1, \dots , x_n \} $.

\begin{lemma}\label{notPossible}
For $\sigma, \tau \in H$ and integers $p,q$ with $1 \leq p \leq
\frac{n}{2}-1$ and $\frac{n}{2} < q \leq n-1$,
$x_{\sigma(p)}x_{\sigma (p+1)}\neq x_{\tau(q)}x_{\tau(q+1)}$.
\end{lemma}
\begin{proof}
Suppose $\sigma(p) = \tau(q)$, i.e. $q=\tau^{-1}\sigma(p)$. We need to show
that  $\tau^{-1}\sigma(p+1) \neq q+1$. Because of Lemma~\ref{other} and the restrictions on $p$ and $q$,
$\tau^{-1} \sigma \notin \langle t \rangle $. Hence, $\tau^{-1} \sigma = t^l u$ for some
integer $l$. Then, because $p\neq \frac{n}{2}$, $\tau^{-1}\sigma(p+1)=t^l u (p+1)=t^l u t (p)= t^{l-1}u(p)=
t^{-1}t^lu(p)=t^{-1}(q)\neq q+1$ because $q\neq n$.
\end{proof}

\begin{lemma}\label{maxOne}
Suppose $1 \leq i < \frac{n}{2}-1$ and let $j$ be an integer such that $i\leq j\leq n$. If $\sigma
,\tau \in H$ are such that
$$ x_{\sigma(n-j+i)} x_{\sigma(n-j+i+1)} \cdots x_{\sigma(n)}= x_{\tau(i)}  x_{\tau(i+1)} \cdots x_{\tau(j)},$$
then either $j=i$, or $j=n$ and $\sigma=\tau$.
\end{lemma}
\begin{proof}
If $j= n$, then the equality of the words yields $\sigma(n)= \tau(n)$ and thus,
by Lemma~\ref{stabilizerr}, $\sigma =\tau$.
So suppose $j\neq n$. First we deal with the case that
$j\geq \frac{n}{2}+i$. Then $\sigma(\frac{n}{2})=\tau(j-\frac{n}{2})$ and $\sigma(\frac{n}{2}+1)=\tau(j-\frac{n}{2}+1)$.
Hence $\sigma^{-1}\tau(j-\frac{n}{2})= \frac{n}{2}$ and $\frac{n}{2}<j \leq n$.
So, by Lemma \ref{other}, $\sigma^{-1}\tau \in \langle t \rangle $. However,
$\sigma^{-1}\tau (j-\frac{n}{2}+1)= \frac{n}{2}+1 $
then implies that $j-\frac{n}{2} + 1 \geq \frac{n}{2} + 1$.
Thus $j \geq n$, a contradiction if $j\neq n$.

So, we have shown that if $j \neq n$, then  $j< \frac{n}2 +i$, and
then $\frac{n}2 < n-j+i <n$.  Suppose that $j\neq n$ and $i\neq
j$. Now, the given equality of words implies
$$ x_{\sigma(n-j+i)} x_{\sigma(n-j+i+1)}= x_{\tau(i)}x_{\tau(i+1)}.$$
Since $\frac{n}2 < n-j +i <n$ and $1\leq i \leq \frac{n}2 -1$
this yields a contradiction with Lemma \ref{notPossible}.
\end{proof}

\begin{lemma}\label{big}
Let $i,j$ be  integers such that $1 \leq i,j \leq \frac{n}2$. If $\sigma
,\tau \in H$ are such that $$ x_{\sigma(j)} x_{\sigma(j+1)} \cdots
x_{\sigma(j+n/2)}= x_{\tau(i)}  x_{\tau(i+1)} \cdots
x_{\tau(i+n/2)},$$ then  $j=i$ and $\sigma=\tau$.
\end{lemma}
\begin{proof} Suppose that $i\neq j$. We may assume that $i<j$.
Since $\tau^{-1}\sigma(j)=i$ and because $1 \leq i,j \leq \frac{n}2$,
by Lemma~\ref{other}, $\tau^{-1}\sigma \in \langle t \rangle$.
Hence $\frac{n}2 +1 \leq \tau^{-1}\sigma(\frac{n}2 +1) \leq n$.
However, the given equality of words implies that
$\tau^{-1}\sigma(\frac{n}2+1)=\frac{n}2+1-j+i\leq \frac{n}2$, a contradiction.
Hence $i=j$ and, by Lemma~\ref{stabilizerr}, $\sigma=\tau$.
\end{proof}

One of the important interpretations of Lemma~\ref{maxOne} concerns the rewriting of an element
$$  a_{m_1}\cdots a_{m_r}
\overbracket{
 a_{i_1} \cdots   
\makebox[0pt][l]{$\displaystyle{\underbracket{\phantom{
a_{i_{n-j+1}} \cdots a_{i_n} a_{l_1} \cdots a_{l_{n-j}} 
} }_{\text{   } } } $}   
a_{i_{n-j+1}} \cdots a_{i_n}                                                 
}^{\text{   }}    
a_{l_1} \cdots a_{l_{n-j}}
\cdots a_{l_s}$$   
of $S_n$ using a relation (as indicated)
\begin{equation*}
\begin{array}{r@{}l}
x_{i_1} \cdots x_{i_n} &{} = x_{\sigma(1)} \cdots x_{\sigma(n)} \vspace{3 pt} \\

x_{i_{n-j+1}} \cdots x_{i_n}x_{l_1} \cdots x_{l_{n-j}}   &{} = x_{\tau(1)} \cdots x_{\tau(n)} .
\end{array}
\end{equation*}
Lemma~\ref{maxOne} says that the number $j$ of overlapping letters of the first word with the second word is at most $1$.

Let $\pi: \FM_n \to S_n$ denote the unique morphism such that $\pi(x_i)=
a_i$ for all $i= 1, \ldots , n$.

\begin{lemma}\label{stepss}
Let $w_1, w_2 \in \FM_n$ be such that $\pi(w_1) = \pi(w_2)$,
$$ w_1= x_{i_1} \cdots x_{i_r} , \quad  w_2= x_{j_1} \cdots x_{j_r} ,
\quad\mbox{and} \quad  i_1 \neq j_1.$$
Then $r\geq n$ and there exist $\sigma, \tau \in H$ such that
$$w_1 = x_{\sigma(1)}\cdots x_{\sigma(n-1)} x_{i_n} \cdots x_{i_r} \quad \mbox{and} \quad
w_2 = x_{\tau(1)}\cdots x_{\tau(n-1)} x_{j_n} \cdots x_{j_r}.$$
Furthermore, exactly one of the following conditions holds:
\vspace{6 pt}
\begin{itemize}
\item[(1)] $\sigma(n) = i_n$ and $ \tau(n) = j_n$,
\item[(2)] $\sigma(n) = i_n$ and $ \tau(n) \neq j_n$,
\item[(3)] $\sigma(n) \neq i_n$ and $ \tau(n) = j_n$.
\end{itemize}
\end{lemma}
\begin{proof}
Because the defining relations are homogeneous of degree $n$ and $i_1 \neq j_1$, it is clear that $r \geq n$.
Furthermore, both words $w_1$ and $w_2$ have to be rewritten using the defining relations. As explained above,
the overlap in such rewrites is of length at most one. It follows that
$w_1 = x_{\sigma(1)}\cdots x_{\sigma(n-1)} x_{i_n} \cdots x_{i_r}$ and
$w_2 = x_{\tau(1)}\cdots x_{\tau(n-1)} x_{j_n} \cdots x_{j_r}$ for some $\sigma, \tau \in H$.
This proves the first part.

To prove the second part, suppose that $\sigma(n) \neq i_n$ and $\tau(n) \neq j_n$.
Because $\pi(w_1)=\pi(w_2)$ and because of the type of defining relations we then get that
there exist
$k_n, k_{n+1}, \ldots, k_r \in \{1,2,\ldots,n\}$ such that
$\pi(x_{i_n} \cdots x_{i_r}) = \pi(x_{k_n} \cdots x_{k_r})$
and $k_n= \sigma(n)$. By the first part of the lemma, $r-n+1 \geq n$ and there exist $\gamma, \delta \in H$ such that
$$ x_{i_n} \cdots x_{i_r} = x_{\gamma(1)} \cdots x_{\gamma(n-1)} x_{i_{2n-1}} \cdots x_{i_r}$$
and
 $$x_{k_n} \cdots x_{k_r} = x_{\delta(1)} \cdots x_{\delta(n-1)} x_{k_{2n-1}} \cdots x_{k_r} .$$
Let $w_1'= x_{\sigma(1)}\cdots x_{\sigma(n-1)} x_{\delta(1)} \cdots x_{\delta(n-1)} x_{k_{2n-1}} \cdots x_{k_r}$.
We have that
$$ \pi(w_2) = \pi(w_1)= \pi(w_1')= \pi (x_{\tau(1)}\cdots x_{\tau(n-1)} x_{\tau \sigma^{-1}\delta(1)} x_{\delta(2)} \cdots x_{\delta(n-1)} x_{k_{2n-1}} \cdots x_{k_r}).$$
By Proposition~\ref{adjannn}, $S_n$ is cancellative, hence it follows that
$$ \pi(x_{\tau \sigma^{-1}\delta(1)} x_{\delta(2)} \cdots x_{\delta(n-1)} x_{k_{2n-1}} \cdots x_{k_r})
= \pi(x_{j_n} \cdots x_{j_r}).$$
Since $\tau \sigma^{-1} \delta (1)= \tau \sigma^{-1}(k_n) = \tau(n)\neq j_n$, by the first part of the lemma there exist
$\zeta, \eta \in H$ such that
\begin{eqnarray*}
 \left( \tau \sigma^{-1} \delta(1), \delta(2), \ldots , \delta(n-1)\right) &=&\left( \zeta(1), \ldots, \zeta(n-1)\right)  \quad \mbox{and} \\
 \left(j_n, j_{n+1}, \ldots , j_{2n-2}\right) &=& \left(\eta (1), \ldots , \eta(n-1)\right).
\end{eqnarray*}
Therefore, by Lemma \ref{stabilizerr}, $\tau \sigma^{-1}\delta = \zeta$ and $\delta = \zeta$.
Hence $\tau = \sigma$. But then $i_1= \sigma(1) = \tau(1) = j_1$, a contradiction.
Thus the result follows.
\end{proof}

\begin{lemma}\label{step3}
Suppose $1\leq i < n$, $\tau\in H$ and $w_1, w_2 \in \FM_n$. If
$$ \pi(w_1)= \pi(x_{\tau(i+1)} \cdots x_{\tau(n)}w_2),$$
then either $w_1 = x_{\tau(i+1)} \cdots x_{\tau(n)}w_2'$, or there
exists $\gamma \in H$ such that  $w_1 = x_{\tau(i+1)} \dots
x_{\tau(n-1)}x_{\gamma(1)} \dots x_{\gamma(n-1)}w_2''$, for some
$w_2', w_2'' \in \FM_n$.
\end{lemma}
\begin{proof}
Let $w_1 = x_{i_1} \cdots x_{i_r}$. Suppose that there exists an integer $s$ such that
$1 \leq s \leq n-i-1$ and
$$ i_1 = \tau(i+1), \  i_2=\tau(i+2), \ \ldots , \ i_{s-1}=\tau(i+s-1) \quad
\mbox{and} \quad i_s \neq \tau(i+s).$$
Then because $S_n$ is cancellative by Proposition~\ref{adjannn}, $\pi(x_{i_s} \cdots x_{i_r})=\pi(x_{\tau(i+s)} \cdots x_{\tau(n)}w_2)$.
So by Lemma \ref{stepss}, $r-s+1 \geq n$ and there exist $\gamma, \delta \in H$ and $w\in \FM_n$ such that
$$ x_{i_s} \cdots x_{i_r} = x_{\gamma(1)} \cdots x_{\gamma(n-1)}x_{i_{s+n-1}}\cdots x_{i_r}
\quad \mbox{and} \quad x_{\tau(i+s)} \cdots x_{\tau(n)}w_2 =
x_{\delta(1)} \cdots x_{\delta(n-1)}w.$$ Since $2\leq i+s < n$ and
$x_{\tau(i+s)} \cdots x_{\tau(n)}=x_{\delta(1)} \cdots
x_{\delta(n-i-s+1)}$, we get a contradiction with
Lemma~\ref{maxOne}. Therefore $$ i_1 = \tau(i+1), \ i_2=\tau(i+2), \
\ldots , \ i_{n-i-1}=\tau(n-1).$$ Now the result follows by the
cancellativity of $S_n$ and Lemma \ref{stepss}.
\end{proof}

 \begin{lemma}\label{overlapp}
Suppose $i,j,l, m$ are integers such that $1\leq j\leq l<m\leq n$, $1\leq i\leq 2$ and $n-1\leq m$. Let
 $\sigma,\tau,\lambda\in H$ be such that $\sigma\neq \tau$. If
$$x_{\sigma(j)} \cdots x_{\sigma(l)}x_{\tau(l+1)} \cdots x_{\tau(m)}=x_{\lambda(i)}\cdots
x_{\lambda(m-j+i)},$$ then $j=l$ and $l+1=m$.
\end{lemma}
\begin{proof}
Because $\sigma \neq \tau$, Lemma~\ref{stabilizerr} yields that
$j\neq i$. We will first prove that $j=l$. We show this by contradiction. So, suppose that $j\neq l$.

Suppose first that $l < \frac{n}2 $. Since $m-1= \tau^{-1}
\lambda(m-1-j+i)$ and $j < l < \frac{n}2$, by Lemma~\ref{other}
$\tau^{-1}\lambda \in \langle t \rangle$.
The assumption yields that $\frac{n}2+1=\tau^{-1} \lambda(\frac{n}2+1 -j+i)$ and thus that $i-j=1$. Hence,
$\tau^{-1}\lambda(\frac{n}2 +1)=\frac{n}2$, which contradicts with $\tau^{-1}\lambda \in \langle t \rangle$.
So $l \geq \frac{n}2$.

Suppose that $l=\frac{n}2$ and $j=1$. Then $i=2$,  $ \sigma(1) =
\lambda(2)$ and $ \sigma(\frac{n}2)= \lambda(\frac{n}2+1)$, in contradiction
with Corollary~\ref{disjoi}.

Suppose that $l=\frac{n}2$ and $j>1$. Since $m= \tau^{-1} \lambda(m-j+i)$,
$\tau^{-1}\lambda \in \langle t \rangle $ by Lemma~\ref{other}.
But this is in contradiction with $\frac{n}2+1= \tau^{-1} \lambda(\frac{n}2+1 -j+i)$
because $j>1$ and $j\neq i$.

So we have shown that $l > \frac{n}2$. Suppose now
that $j \leq \frac{n}2$. Since $j = \sigma^{-1}\lambda(1)$ or $j = \sigma^{-1}\lambda(2)$,
Lemma~\ref{other} then yields that $\sigma^{-1}\lambda \in \langle t \rangle$.
Then $\frac{n}2+1 = \sigma^{-1}\lambda(\frac{n}2+1-j+i)$  implies that
$i-j=1$. Hence we have, $\frac{n}2=\sigma^{-1}\lambda(\frac{n}2+1)$, which is
not possible. So we have obtained that
$j \geq \frac{n}2+1$. Since $\frac{n}2<j<l$, we get a contradiction with
Lemma~\ref{notPossible}. Therefore $j=l$.

Next we prove that $l+1=m$. Again we show this by contradiction. So
suppose that $l+1\neq m$.

Suppose that $l=1$. Since $\sigma \neq \tau$ the assumption and Lemma~\ref{stabilizerr} yield that $i\neq 1$
and thus $i=2$. Consequently, since $j=l$,
$\tau(2)=\lambda(3)$
and $\tau(\frac{n}2)=\lambda(\frac{n}2+1)$, in contradiction with
Corollary~\ref{disjoi}.

Suppose that $1<l<\frac{n}2$. Since $m= \tau^{-1} \lambda(m-j+i)$ and $j =
l < \frac{n}2$, we obtain from Lemma~\ref{other} that $\tau^{-1}\lambda \in \langle t \rangle $.
However, this is in contradiction with $\frac{n}2+1= \tau^{-1} \lambda(\frac{n}2+1 -j+i)$
because $j=l>1$ and $j\neq i$.
So we have shown that $j=l \geq \frac{n}2$. Since $\frac{n}2<l+1<m$,
the assumption leads to a contradiction with Lemma~\ref{notPossible}. Therefore $l+1=m$, as desired
\end{proof}

There are obvious symmetric analogs of
Lemmas~\ref{maxOne}, \ref{stepss}, \ref{step3} and~\ref{overlapp}.
The statement of the symmetric version is obtained in the following way.
\begin{itemize}
 \item Words are written in reversed order.
 \item An integer $i$ which is an index of a generator of $\FM_n$ or which refers to such an index
 is replaced by $j \in \{1, 2, \ldots , n \}$,
 such that $j \equiv -i+1 \ (\mathrm{mod} \ n)$.
 \item Inequalities involving integers which refer to indices of generators of $\FM_n$ are reversed.
\end{itemize}

When these lemmas are used, it is clear
from the context which of the two versions is applicable.

\section{Main Theorem}

In this section, we prove that $S_n(H)$ is a two unique product semigroup. 
As a consequence, $K[S_n(H)]$ is a domain with trivial units which is semiprimitive.

\begin{theorem}\label{tup}
 $S_n$ is a t.u.p. semigroup.
\end{theorem}
\begin{proof}
Let $C,D \subseteq S_n$ be nonempty sets such that $|C|+|D|>2$.

Since the defining relations are homogeneous of length $n$, it is clear that if both $C$ and $D$ have
a unique element of shortest and longest length, then $CD$ has two uniquely presented elements.

Hence, to prove the result, without loss of generality we may assume that all elements of $C$ have
the same length, and also that all elements of $D$ have the same length.

Suppose $cd = c'd'$, with $c,c' \in C$, $d,d'\in D$, and $c\neq c'$ or $d\neq d'$.
Since $S_n$ is cancellative by Proposition~\ref{adjannn}, we actually have
that $c\neq c'$ and $d\neq d'$.

Choose $w_1 \in \pi^{-1}(c)$ and
$w_2 \in \pi^{-1}(c')$ and write in $\FM_n = \langle x_1, \ldots , x_2 \rangle$:
$$w_1=x_{i_1} \cdots  x_{i_r}, \quad  w_2=x_{j_1} \cdots x_{j_r}, \quad
i_1 = j_1, \ldots , i_{k-1} = j_{k-1} , \quad i_k \neq j_k, \quad k \leq r, $$
where $k$ is maximal. \\
Let $x_{i_{r+1}}\cdots x_{i_{r+s}} \in \pi^{-1}(d)$ and $x_{j_{r+1}} \cdots x_{j_{r+s}} \in \pi^{-1}(d')$.
Since
$$\pi(x_{i_1}\cdots x_{i_r}x_{i_{r+1}}\cdots x_{i_{r+s}})=\pi(x_{j_1}\cdots x_{j_r}x_{j_{r+1}}\cdots x_{j_{r+s}})$$
and $S_n$ is cancellative, we have that
$$\pi(x_{i_k}\cdots x_{i_r}x_{i_{r+1}}\cdots x_{i_{r+s}})=\pi(x_{j_k}\cdots x_{j_r}x_{j_{r+1}}\cdots x_{j_{r+s}}).$$
By Lemma \ref{stepss}, $r+s-k+1\geq n$ and there exist distinct $\sigma, \tau \in H$ such that
\begin{equation*}
\begin{array}{r@{}l}
x_{i_k}\cdots x_{i_r}x_{i_{r+1}}\cdots x_{i_{r+s}} &{}= x_{\sigma(1)} \cdots x_{\sigma(n-1)}x_{i_{k+n-1}}\cdots x_{i_{r+s}}, \vspace{3 pt} \\
 x_{j_k} \cdots x_{j_r}x_{j_{r+1}}\cdots x_{j_{r+s}} &{}= x_{\tau(1)} \cdots x_{\tau(n-1)}x_{j_{k+n-1}}\cdots x_{j_{r+s}}.
\end{array}
\end{equation*}
Furthermore, one of the three conditions of Lemma~\ref{stepss} holds.
We claim that $r-k <n-1$. We prove this by contradiction.
So suppose that $r-k\geq n-1$. \\
If condition (1) of Lemma~\ref{stepss} holds, then
\begin{equation*}
\begin{array}{r@{}l}
c&{}= \pi(x_{i_1} \cdots x_{i_r})= \pi(x_{i_1} \cdots x_{i_{k-1}} x_1 \cdots x_n x_{i_{k+n}} \cdots x_{i_r}),  \vspace{3 pt} \\
c'&{}= \pi(x_{j_1} \cdots x_{j_r})= \pi(x_{i_1} \cdots x_{i_{k-1}} x_1 \cdots x_n x_{j_{k+n}} \cdots x_{j_r}).
\end{array}
\end{equation*}
In contradiction with the maximality of $k$. \\
If condition (2) of Lemma~\ref{stepss} holds, then
$$ c= \pi(x_{i_1} \cdots x_{i_r})= \pi(x_{i_1} \cdots x_{i_{k-1}} x_{\tau(1)} \cdots x_{\tau(n)} x_{i_{k+n}} \cdots x_{i_r})$$
and $c'= \pi(x_{j_1}\cdots x_{j_r})$. Since $(i_1, \cdots, i_{k-1}, \tau(1)) =(j_1, \cdots, i_{k-1}, j_k)$, we again get
a contradiction with the maximality of $k$.\\
If condition (3) of Lemma~\ref{stepss} holds, then $c= \pi(x_{i_1}\cdots x_{i_r})$ and
$$ c'=\pi(x_{j_1} \cdots x_{j_r})= \pi(x_{i_1} \cdots x_{i_{k-1}} x_{\sigma(1)} \cdots x_{\sigma(n)} x_{j_{k+n}} \cdots x_{j_r})$$
and we also get a contradiction with the maximality of $k$.
Therefore the claim follows.

Hence $r-k < n-1$ and
$$c= \pi(x_{i_1}\cdots x_{i_{k-1}}  x_{\sigma(1)} \cdots x_{\sigma(r-k+1)}) \quad \mbox{and} \quad
c'= \pi(x_{i_1}\cdots x_{i_{k-1}}  x_{\tau(1)} \cdots x_{\tau(r-k+1)})$$
with $\sigma \neq \tau$. Now using cancellativity and Lemma~\ref{stepss}, it is easy to see that we may assume
that
$$ d= \pi(x_{\sigma(r-k+2)} \cdots x_{\sigma(n)} x_{l_{k+n}} \cdots x_{l_{r+s}})$$
and
$$ d'= \pi(x_{\tau (r-k+2)} \cdots x_{\tau(n)} x_{l_{k+n}} \cdots x_{l_{r+s}})$$
for some $ l_{k+n}, \ \ldots , \ l_{r+s} \in \{1, \ldots ,n \}$.

If $cd'$ and $c'd$ both have a unique presentation
in $CD$ then we are finished.
So assume that one of them does not have a unique presentation in $CD$.
Say $cd'$ does not have a unique presentation. An argument symmetric to the one we shall present
will work for the other case.
Then there exist $c'' \in C$ and $d''\in D$ such that $cd'=c''d''$,
$c''\neq c$ and $d'\neq d''$.

By the method used above, there exist $\lambda, \mu \in H$ and $r-n+1 < k' \leq r$
such that $\lambda \neq \mu$ and
\begin{equation*}
\begin{array}{r@{}l}
c&{}=  \pi(x_{i_1} \cdots x_{i_{k-1}} x_{\sigma(1)} \cdots x_{\sigma(r-k+1)})= \pi(x_{i'_1} \cdots x_{i'_{k'-1}} x_{\lambda(1)} \cdots x_{\lambda(r-k'+1)}), \vspace{3 pt} \\
d'&{}=   \pi(x_{\tau (r-k+2)} \cdots x_{\tau(n)} x_{l_{k+n}} \cdots x_{l_{r+s}}) =
\pi(x_{\lambda (r-k'+2)} \cdots x_{\lambda(n)} x_{l'_{k'+n}} \cdots x_{l'_{r+s}})  ,  \vspace{3 pt} \\
c''&{}=  \pi(x_{i'_1} \cdots x_{i'_{k'-1}} x_{\mu(1)} \cdots x_{\mu(r-k'+1)}) ,  \vspace{3 pt} \\
d''&{}=  \pi(x_{\mu (r-k'+2)} \cdots x_{\mu(n)} x_{l'_{k'+n}} \cdots x_{l'_{r+s}}).
\end{array}
\end{equation*}
for some $i'_1,\ \ldots , \ i'_{k'-1}, \  l'_{k'+n}, \ \ldots , \ l'_{r+s} \in \{1, \ldots ,n \}$.

We claim that either $k'=r-n+2$ and $k=r$, or $k=r-n+2$ and $k'=r$.

Let $v,w'\in \FM_n$ be such that $\pi(v)=c$ and $\pi(w')=d'$. By Lemma \ref{step3} and its right-left
symmetric analog, we know that
for some $\V_1,\V_2,\V_3,\V_4,\W'_1,\W'_2,\W'_3,\W'_4\in \FM_n$ and $\epsilon, \delta, \gamma, \zeta \in H$, we have the following four statements
\begin{equation}\label{fourStatements}
\left\{
\begin{array}{r@{}l@{}r@{}r@{}l}
v&{}=\V_1x_{\sigma(1)} \cdots x_{\sigma(r-k+1)} &{}\hspace{13 pt}  \text{or} \hspace{13 pt}    &{} v&{}=\V_2 x_{\epsilon(2)} \cdots x_{\epsilon(n)} x_{\sigma(2)} \cdots x_{\sigma(r-k+1)}   \\
v&{}=\V_3x_{\lambda(1)} \cdots x_{\lambda(r-k'+1)} &{} \hspace{13 pt}  \text{or} \hspace{13 pt}  &{} v&{}=\V_4 x_{\delta(2)} \cdots x_{\delta(n)} x_{\lambda(2)} \cdots x_{\lambda(r-k'+1)}   \\
w'&{}=x_{\tau (r-k+2)} \cdots x_{\tau(n)}\W'_1 &{} \hspace{13 pt}  \text{or} \hspace{13 pt}&{}  w'&{}=x_{\tau (r-k+2)} \cdots x_{\tau(n-1)}x_{\gamma(1)} \cdots x_{\gamma(n-1)} \W'_2  \\
w'&{}=x_{\lambda (r-k'+2)} \cdots x_{\lambda(n)}\W'_3 &{}\hspace{13 pt}  \text{or} \hspace{13 pt}  &{} w'&{}=x_{\lambda (r-k'+2)} \cdots x_{\lambda(n-1)}x_{\zeta(1)} \cdots x_{\zeta(n-1)} \W'_4.
\end{array} \right.
\end{equation}
We consider two mutually exclusive cases.

{\it Case 1:} Suppose that none of the following four conditions (A), (B), (C) or (D) is satisfied.
\begin{itemize}
 \item[(A)] $k=r$, \ $v=\V_2 x_{\epsilon(2)} \cdots x_{\epsilon(n)}$ \ and \ $\epsilon(n) \neq \sigma(1)$
 \item[(B)] $k'=r$, \ $v=\V_4 x_{\delta(2)} \cdots x_{\delta(n)}$ \ and \ $\delta(n) \neq \lambda(1)$
 \item[(C)] $k=r-n+2$, \  $w'= x_{\gamma(1)} \cdots x_{\gamma(n-1)} \W'_2$ \ and \ $\gamma(1) \neq \tau(n)$
 \item[(D)] $k'=r-n+2$, \  $w'= x_{\zeta(1)} \cdots x_{\zeta(n-1)} \W'_4$ \  and \ $\zeta(1) \neq \lambda(n).$
\end{itemize}

First we show that $k\neq k'$. We prove this by contradiction. So assume $k=k'$. Then from (\ref{fourStatements}),
we get that the last letter of $v$ is $x_{\sigma(r-k+1)}= x_{\lambda(r-k+1)}$ and that the first letter of $w'$ is  $x_{\tau (r-k+2)}=x_{\lambda (r-k+2)}$.
So, by Lemma~\ref{stabilizerr}, $\sigma = \lambda = \tau$, a contradiction. Thus indeed $k\neq k'$. Suppose $k' > k$.

Because of (\ref{fourStatements}), one of the following four equalities holds
\begin{eqnarray}
vw'& \hspace{-9 pt}  =& \hspace{-9 pt} \V_1x_{\sigma(1)} \cdots x_{\sigma(r-k+1)} x_{\tau (r-k+2)} \cdots x_{\tau(n)}\W'_1  \label{1of8}   \vspace{3 pt} \\
 vw'& \hspace{-9 pt}  =& \hspace{-9 pt}\V_1x_{\sigma(1)} \cdots x_{\sigma(r-k+1)} x_{\tau (r-k+2)} \cdots x_{\tau(n-1)}x_{\gamma(1)} \cdots x_{\gamma(n-1)} \W'_2 \label{2of8}  \vspace{3 pt} \\
 vw'& \hspace{-9 pt}  =& \hspace{-9 pt}\V_2 x_{\epsilon(2)} \cdots x_{\epsilon(n)} x_{\sigma(2)} \cdots x_{\sigma(r-k+1)} x_{\tau (r-k+2)} \cdots x_{\tau(n)}\W'_1 \label{3of8}  \vspace{3 pt} \\
 vw'& \hspace{-9 pt}  =& \hspace{-9 pt} \V_2 x_{\epsilon(2)} \cdots x_{\epsilon(n)} x_{\sigma(2)} \cdots x_{\sigma(r-k+1)}
 x_{\tau (r-k+2)} \cdots x_{\tau(n-1)}x_{\gamma(1)} \cdots x_{\gamma(n-1)} \W'_2, \phantom{ ssssss } \label{4of8}
\end{eqnarray}
and also one of the following four equalities holds
\begin{eqnarray}
vw'& \hspace{-9 pt}  =& \hspace{-9 pt} \V_3x_{\lambda(1)} \cdots x_{\lambda(r-k'+1)}x_{\lambda (r-k'+2)} \cdots x_{\lambda(n)}\W'_3  \label{5of8}   \vspace{3 pt} \\
 vw'& \hspace{-9 pt}  =& \hspace{-9 pt}\V_3x_{\lambda(1)} \cdots x_{\lambda(r-k'+1)}x_{\lambda (r-k'+2)} \cdots x_{\lambda(n-1)}x_{\zeta(1)} \cdots x_{\zeta(n-1)} \W'_4 \label{6of8}  \vspace{3 pt} \\
 vw'& \hspace{-9 pt}  =& \hspace{-9 pt}\V_4 x_{\delta(2)} \cdots x_{\delta(n)} x_{\lambda(2)} \cdots x_{\lambda(r-k'+1)}x_{\lambda (r-k'+2)} \cdots x_{\lambda(n)}\W'_3 \label{7of8}  \vspace{3 pt} \\
 vw'& \hspace{-9 pt}  =& \hspace{-9 pt} \V_4 x_{\delta(2)} \cdots x_{\delta(n)} x_{\lambda(2)} \cdots
x_{\lambda(r-k'+1)}x_{\lambda (r-k'+2)} \cdots x_{\lambda(n-1)}x_{\zeta(1)} \cdots x_{\zeta(n-1)} \W'_4. \phantom{ ssssss } \label{8of8}
\end{eqnarray}

There are 16 possible combinations, we treat them in groups of 4. By $|v|$ we denote the length of the word $v$.

In case (\ref{1of8}) and (\ref{5of8}), or  (\ref{1of8}) and (\ref{6of8}), or  (\ref{3of8}) and (\ref{5of8}), or  (\ref{3of8}) and (\ref{6of8})
hold, we take the subword of $vw'$ starting at position $|v|-(r-k')$ until position
$|v|+n-(r-k+1)$. By inspection, one sees that for some $j\geq 2$,
$$ x_{\sigma(j)} \cdots x_{\sigma(r-k+1)}x_{\tau (r-k+2)} \cdots x_{\tau(n)}
= x_{\lambda(1)} \cdots x_{\lambda(n-j+1)}. $$
By Lemma~\ref{overlapp}, $j=r-k+1$ and $n=r-k+2$. Thus $k=r-n+2$.
The first letter of $w'$ is $x_{\lambda (r-k'+2)} = x_{\lambda(n-j+1)} = x_{\lambda(2)}$.
Thus $k'=r$, as claimed.

In case (\ref{1of8}) and (\ref{7of8}), or  (\ref{1of8}) and (\ref{8of8}), or  (\ref{3of8}) and (\ref{7of8}), or  (\ref{3of8}) and (\ref{8of8})
hold, we take the subword of $vw'$ starting at position $|v|-(r-k'-1)$ until position
$|v|+n-(r-k+1)$. By inspection, one sees that for some $j\geq 3$,
$$ x_{\sigma(j)} \cdots x_{\sigma(r-k+1)}x_{\tau (r-k+2)} \cdots x_{\tau(n)}
= x_{\lambda(2)} \cdots x_{\lambda(n-j+2)}. $$
By Lemma~\ref{overlapp}, $j=r-k+1$ and $n=r-k+2$.
The first letter of $w'$ is $x_{\lambda (r-k'+2)} = x_{\lambda(n-j+2)} = x_{\lambda(3)}$.
Thus $k'=r-1$. Then
\begin{eqnarray*}
&&v= \V_1x_{\sigma(1)} \cdots x_{\sigma(n-1)}= \V_4 x_{\delta(2)} \cdots x_{\delta(n)} x_{\lambda(2)} \\
\text{or} \ \ \
&&v= \V_2x_{\epsilon(2)} \cdots x_{\epsilon(n)}x_{\sigma(2)} \cdots x_{\sigma(n-1)}
= \V_4 x_{\delta(2)} \cdots x_{\delta(n)} x_{\lambda(2)}
\end{eqnarray*}
would lead to a contradiction with Lemma \ref{big}.

In case (\ref{2of8}) and (\ref{5of8}), or  (\ref{2of8}) and (\ref{6of8}), or  (\ref{4of8}) and (\ref{5of8}), or  (\ref{4of8}) and (\ref{6of8})
hold, we take the subword of $vw'$ starting at position $|v|-(r-k')$ until position
$|v|+n-(r-k+2)$. By inspection, one sees that for some $j\geq 2$,
$$ x_{\sigma(j)} \cdots x_{\sigma(r-k+1)}x_{\tau (r-k+2)} \cdots x_{\tau(n-1)}
= x_{\lambda(1)} \cdots x_{\lambda(n-j)}. $$
By Lemma~\ref{overlapp}, $j=r-k+1$ and $n=r-k+3$.
The first letter of $w'$ is $x_{\lambda (r-k'+2)} = x_{\lambda(n-j)} = x_{\lambda(2)}$.
Thus $k'=r$. Then
\begin{eqnarray*}
&&w'= x_{\tau(n-1)}x_{\gamma(1)} \cdots x_{\gamma(n-1)} \W'_2 = x_{\lambda (2)} \cdots x_{\lambda(n)}\W'_3 \\
\text{or} \ \ \
&&w'= x_{\tau(n-1)}x_{\gamma(1)} \cdots x_{\gamma(n-1)} \W'_2 = x_{\lambda (2)} \cdots x_{\lambda(n-1)}x_{\zeta(1)} \cdots x_{\zeta(n-1)} \W'_4
\end{eqnarray*}
would lead to a contradiction with Lemma \ref{big}.

In case (\ref{2of8}) and (\ref{7of8}), or  (\ref{2of8}) and (\ref{8of8}), or  (\ref{4of8}) and (\ref{7of8}), or  (\ref{4of8}) and (\ref{8of8})
hold, we take the subword of $vw'$ starting at position $|v|-(r-k'-1)$ until position
$|v|+n-(r-k+2)$. By inspection, one sees that for some $j\geq 3$,
$$ x_{\sigma(j)} \cdots x_{\sigma(r-k+1)}x_{\tau (r-k+2)} \cdots x_{\tau(n-1)}
= x_{\lambda(2)} \cdots x_{\lambda(n-j)}. $$
By Lemma~\ref{overlapp}, $j=r-k+1$ and $n=r-k+3$.
The first letter of $w'$ is $x_{\lambda (r-k'+2)} = x_{\lambda(n-j)} = x_{\lambda(3)}$.
Thus $k'=r-1$. Then
\begin{eqnarray*}
&&v= \V_1x_{\sigma(1)} \cdots x_{\sigma(n-2)}= \V_4 x_{\delta(2)} \cdots x_{\delta(n)} x_{\lambda(2)} \\
\text{or} \ \ \
&&v= \V_2x_{\epsilon(2)} \cdots x_{\epsilon(n)}x_{\sigma(2)} \cdots x_{\sigma(n-2)}
= \V_4 x_{\delta(2)} \cdots x_{\delta(n)} x_{\lambda(2)}
\end{eqnarray*}
would lead to a contradiction with Lemma \ref{big}.

Thus $k=r-n+2$ and $k'=r$. Hence we may assume that (~\ref{fourStatements}) is reduced to
\begin{equation}\label{fourStatementsTwo}
\left\{
\begin{array}{r@{}l@{}r@{}r@{}l}
v&{}=\V_1x_{\sigma(1)} \cdots x_{\sigma(n-1)} &{} \hspace{13 pt}  \text{or} \hspace{13 pt}  &{}  v&{}=\V_2 x_{\epsilon(2)} \cdots x_{\epsilon(n)} x_{\sigma(2)} \cdots x_{\sigma(n-1)} \\
v&{}=\V_3x_{\lambda(1)} &{} &{} &{}  \\
w'&{}= x_{\tau(n)}\W'_1 &{} &{} &{}   \\
w'&{}=x_{\lambda (2)} \cdots x_{\lambda(n)}\W'_3 &{} \hspace{13 pt}  \text{or} \hspace{13 pt}  &{}  w'&{}=x_{\lambda (2)} \cdots x_{\lambda(n-1)}x_{\zeta(1)} \cdots x_{\zeta(n-1)} \W'_4.
\end{array} \right.
\end{equation}
Let $v',w,v'',w''\in \FM_n$ be such that $\pi(v')=c'$, $\pi(w)=d$, $\pi(v'')=c''$ and $\pi(w'')=d''$.
Since $k=r-n+2$ and $k'=r$,
by (\ref{fourStatementsTwo}) and Lemma \ref{step3}  we know that
for some $\V'_1, \V'_2, \W_1, \W_2, \V''_1,\V''_2,\W''_1,\W''_2 \in \FM_n$ and $\eta, \theta, \iota, \kappa \in H$, the following four statements hold
\begin{equation*}
\left\{
\begin{array}{r@{}l@{}r@{}r@{}l}
v'&{}=\V'_1x_{\tau(1)} \cdots x_{\tau(n-1)}  &{}\hspace{13 pt}  \text{or} \hspace{13 pt}  &{} v'&{}=\V'_2 x_{\eta(2)} \cdots x_{\eta(n)}x_{\tau(2)} \cdots x_{\tau(n-1)}    \\
w&{}=x_{\sigma(n)}x_{\lambda (3)} \cdots x_{\lambda(n)}\W_1 &{} \hspace{13 pt}  \text{or} \hspace{13 pt}   &{} w&{}=x_{\sigma(n)}x_{\lambda (3)} \cdots x_{\lambda(n-1)} x_{\theta(1)} \cdots x_{\theta(n-1)} \W_2  \\
v''&{}=\V''_1x_{\sigma(1)} \cdots x_{\sigma(n-2)} x_{\mu(1)}  &{} \hspace{13 pt}  \text{or} \hspace{13 pt}  &{} v''&{}=\V''_2 x_{\iota(2)} \cdots x_{\iota(n)} x_{\sigma(2)} \cdots x_{\sigma(n-2)} x_{\mu(1)}   \\
w''&{}=x_{\mu (2)} \cdots x_{\mu(n)}\W''_1 &{} \hspace{13 pt}  \text{or} \hspace{13 pt}   &{} w''&=x_{\mu (2)} \cdots x_{\mu(n-1)}x_{\kappa(1)} \cdots x_{\kappa(n-1)} \W''_2.
\end{array} \right.
\end{equation*}
Because $\tau \neq \sigma$, by Lemma \ref{stabilizerr} we have $x_{\sigma(n)} \neq x_{\tau(n)}= x_{\lambda(2)}$,
and by Lemma \ref{big} it follows that $c'd =\pi(v')\pi(w)$ is uniquely presented.
Because $\mu \neq \lambda$, by Lemma \ref{stabilizerr} we have $x_{\mu(1)} \neq x_{\lambda(1)}$,
and by Lemma \ref{big} it follows that $c''d'$ is uniquely presented.

By symmetry, if $k > k'$, then $k'=r-n+2$ and $k=r$, and the products $c'd$ and $cd''$ are uniquely presented.

{\it Case 2:} Suppose one of the conditions (A), (B), (C), or (D) is satisfied.

(A) Suppose $k=r$, $v=\V_2 x_{\epsilon(2)} \cdots x_{\epsilon(n)}$ and $\epsilon(n) \neq \sigma(1).$
Because $c=\pi(x_{i_1}\cdots x_{i_{k-1}}  x_{\sigma(1)} )=
\pi(\V_2 x_{\epsilon(2)} \cdots x_{\epsilon(n)})$, by Lemmas \ref{stabilizerr} and \ref{maxOne}, for
$\xi \in H$ with $\xi(n)=\sigma(1)$, we have $x_{i_{k-n+2}}\cdots x_{i_{k-1}}
= x_{\xi(2)} \cdots x_{\xi(n-1)}$.

Let $v'\in \FM_n$ be such that $\pi(v')=c'$. Since $c'=\pi(x_{i_1}\cdots x_{i_{k-1}}  x_{\tau(1)} )$
and $\tau \neq \sigma$, by Lemmas \ref{stabilizerr} and \ref{maxOne}
it follows that $v'= \V' x_{i_{k-n+2}}\cdots x_{i_{k-1}}  x_{\tau(1)}
= \V' x_{\xi(2)} \cdots x_{\xi(n-1)}  x_{\tau(1)} $ for $\V' \in \FM$
with $\pi(\V') = \pi(x_{i_1}\cdots x_{i_{k-n+1}})$.

Let $w \in \FM_n$ be such that $\pi(w)=d$. Because $k=r$, by Lemma \ref{maxOne} there exists $\W \in \FM_n$ such
that $w= x_{\sigma(2)} \cdots x_{\sigma(n-1)}\W$.

We then consider the product \\
$\pi(v')\pi(w)=\pi(\V' x_{\xi(2)} \cdots x_{\xi(n-1)}  x_{\tau(1)})\pi(x_{\sigma(2)} \cdots x_{\sigma(n-1)}\W)=c'd$.
Because $\sigma \neq \tau$ and $\xi(n) = \sigma(1)$, by Lemmas \ref{stabilizerr} and \ref{big},
$c'd$ has a unique presentation.

Because $v=\V_2 x_{\epsilon(2)} \cdots x_{\epsilon(n)}$, and
$ v=\V_3x_{\lambda(1)} \cdots x_{\lambda(r-k'+1)}$ or $v=\V_4 x_{\delta(2)} \cdots x_{\delta(n)} x_{\lambda(2)} \cdots x_{\lambda(r-k'+1)},  $
unless $k'=r$ we would get a contradiction with Lemma \ref{maxOne}.

Since $cd'=c''d''$ and $d''\neq d'$,  we know that
$c''=\pi(\V''x_{\theta(2)} \cdots x_{\theta(n-1)}  x_{\mu(1)})$,  for $\V''\in \FM$ and
$\theta \in H$ such that
$\pi(\V'')=\pi(x_{i_1} \cdots x_{i_{k-n}} x_{\theta(1)})$
and $x_{\theta(n)}=x_{\lambda(1)}\neq x_{\mu(1)}$.
Then by Lemmas \ref{stabilizerr} and \ref{big},
$c''d'= \pi(\V''x_{\theta(2)} \cdots x_{\theta(n-1)}  x_{\mu(1)})
 \pi(x_{\lambda (2)} \cdots x_{\lambda(n-1)} \W'_5)$ with $\W'_5\in FM_n$
has a unique presentation.

(B) Suppose $k'=r$, $v=\V_4 x_{\delta(2)} \cdots x_{\delta(n)}$ and $\delta(n) \neq \lambda(1)$.
Because $c=\pi(x_{i_1}\cdots x_{i_{k-1}}  x_{\sigma(1)} )=
\pi(\V_4 x_{\delta(2)} \cdots x_{\delta(n)})$, by Lemmas \ref{stabilizerr} and \ref{maxOne}, for
$\xi \in H$ with $\xi(n)=\sigma(1)$, we have $x_{i_{k-n+2}}\cdots x_{i_{k-1}}
= x_{\xi(2)} \cdots x_{\xi(n-1)}$. As in case (A), we obtain the uniquely presented products
$c'd$ and $c''d'$.

(C) Suppose $k=r-n+2$, $w'= x_{\gamma(1)} \cdots x_{\gamma(n-1)} \W'_2$ and $\gamma(1) \neq \tau(n)$. \\
Because $d'=\pi(x_{\tau(n)} x_{l_{k+n}}   \cdots x_{l_{r+s}})=
\pi(x_{\gamma(1)} \cdots x_{\gamma(n-1)} \W'_2)$, by Lemmas \ref{stabilizerr} and \ref{maxOne}, for
$\phi \in H$ with $\phi(1)=\tau(n)$, we have $x_{l_{k+n}} \cdots x_{l_{k+2n-3}}
= x_{\phi(2)} \cdots x_{\phi(n-1)}$.

Let $w\in \FM_n$ be such that $\pi(w)=d$. Since $d=\pi(x_{\sigma (n)} x_{l_{k+n}}   \cdots x_{l_{r+s}} )$
and $\sigma \neq \tau$, by Lemmas \ref{stabilizerr} and \ref{maxOne}
it follows that $w= x_{\sigma(n)} x_{l_{k+n}} \cdots x_{l_{k+2n-3}} \W
=  x_{\sigma(n)} x_{\phi(2)} \cdots x_{\phi(n-1)} \W $ for $\W \in \FM $
with $\pi(\W) = \pi(x_{l_{k+2n-2}}\cdots x_{l_{r+s}})$.

Let $v'\in \FM_n$ be such that $\pi(v')=c'$. Because $k=r-n+2$, by Lemma \ref{maxOne} there exists $\V' \in \FM_n$ such
that $v'= \V' x_{\tau(2)} \cdots x_{\tau(n-1)}$.

We then consider the product \\
$\pi(v')\pi(w)=\pi( \V' x_{\tau(2)} \cdots x_{\tau(n-1)})\pi(x_{\sigma(n)} x_{\phi(2)} \cdots x_{\phi(n-1)} \W)=c'd$.
Because $\sigma \neq \tau$ and $\phi(1)=\tau(n)$, by Lemmas \ref{stabilizerr} and \ref{big},
$c'd$ has a unique presentation.

Because $w'= x_{\gamma(1)} \cdots x_{\gamma(n-1)} \W'_2$, and  \\
$ w'=x_{\lambda (r-k'+2)} \cdots x_{\lambda(n)}\W'_3$ or $w'=x_{\lambda (r-k'+2)} \cdots x_{\lambda(n-1)}x_{\zeta(1)} \cdots x_{\zeta(n-1)} \W'_4,  $
unless $k'=r-n+2$ we would get a contradiction with Lemma \ref{maxOne}.

Since $cd'=c''d''$ and $c''\neq c$,  we know that
$d''=\pi(x_{\mu (n)} x_{\rho(2)} \cdots x_{\rho(n-1)} \W'')$,  for $\W''\in \FM$ and
$\rho \in H$ such that
$\pi(\W'')=\pi( x_{i_{\rho(n)}} x_{l_{k+2n-1}}\cdots x_{l_{r+s}} )$
and $x_{\rho (1)}=x_{\lambda(n)}\neq x_{\mu(n)}$.
Then by Lemmas \ref{stabilizerr} and \ref{big},
$cd''= \pi(\V_5 x_{\lambda (2)} \cdots x_{\lambda(n-1)})
\pi(x_{\mu (n)} x_{\rho(2)} \cdots x_{\rho(n-1)} \W'')$ with $\V_5 \in \FM_n$
has a unique presentation.

(D) Suppose $k'=r-n+2$, $w'= x_{\zeta(1)} \cdots x_{\zeta(n-1)} \W'_4$ and $\zeta(1) \neq \lambda(n)$ . \\
Because $d'=\pi(x_{\tau(n)} x_{l_{k+n}}   \cdots x_{l_{r+s}})=
\pi(x_{\zeta(1)} \cdots x_{\zeta(n-1)} \W'_4)$, by Lemmas \ref{stabilizerr} and \ref{maxOne}, for
$\phi \in H$ with $\phi(1)=\tau(n)$, we have $x_{l_{k+n}} \cdots x_{l_{k+2n-3}}
= x_{\phi(2)} \cdots x_{\phi(n-1)}$.
As in case (C), we obtain the uniquely presented products
$c'd$ and $cd''$.
\end{proof}

 Let $K$ be a field. Since the defining relations of $S_n$ are
homogeneous with respect to the total degree, the $K$-algebra
$K[S_n]$ is a graded algebra.
\begin{theorem}\label{domai}
$K[S_n(H)]$ is a domain with trivial units. The Jacobson radical
$\mathcal{J}(K[S_n(H)])=0$.
\end{theorem}
\begin{proof}
The result follows from Theorem \ref{tup},
\cite[Theorem~10.4]{book-jan} and
\cite[Corollary~10.5]{book-jan}.
\end{proof}

\vspace{30pt}
 \noindent \begin{tabular}{llllllll}
 F. Ced\'o && E. Jespers  \\
 Departament de Matem\`atiques &&  Department of Mathematics \\
 Universitat Aut\`onoma de Barcelona &&  Vrije Universiteit Brussel  \\
08193 Bellaterra (Barcelona), Spain    && Pleinlaan
2, 1050 Brussel, Belgium \\
 cedo@mat.uab.cat && efjesper@vub.ac.be\\
   &&   \\
G. Klein &&  \\ Department of Mathematics &&
\\  Vrije Universiteit Brussel && \\
Pleinlaan 2, 1050 Brussel, Belgium &&\\ gklein@vub.ac.be&&
\end{tabular}
\end{document}